\documentclass[]{amsart}
\usepackage{longtable}

\usepackage{amscd,amsthm,amssymb,amsfonts,amsmath,euscript}

\theoremstyle{plain}
\newtheorem{thm}{Theorem}[section]
\newtheorem{lemma}[thm]{Lemma}
\newtheorem{prop}[thm]{Proposition}
\newtheorem{cor}[thm]{Corollary}

\theoremstyle{definition}

\theoremstyle{remark}

\newtheorem*{thank}{{\bf Acknowledgments}}
\newcommand{\nc}{\newcommand}

\def\makeop#1{\expandafter\def\csname#1\endcsname
  {\mathop{\rm #1}\nolimits}\ignorespaces}
\makeop{Hom}   \makeop{End}   \makeop{Aut}   \makeop{Isom}  \makeop{Pic} 
\makeop{Gal}   \makeop{ord}   \makeop{Char}  \makeop{Div}   \makeop{Lie} 
\makeop{PGL}   \makeop{Corr}  \makeop{PSL}   \makeop{sgn}   \makeop{Spf}
\makeop{Spec}  \makeop{Tr}    \makeop{Nr}    \makeop{Fr}    \makeop{disc}
\makeop{Proj}  \makeop{supp}  \makeop{ker}   \makeop{im}    \makeop{dom}
\makeop{coker} \makeop{Stab}  \makeop{SO}    \makeop{SL}    \makeop{SL}
\makeop{Cl}    \makeop{cond}  \makeop{Br}    \makeop{inv}   \makeop{rank}
\makeop{id}    \makeop{Fil}   \makeop{Frac}  \makeop{GL}    \makeop{SU}
\makeop{Trd}   \makeop{Sp}    \makeop{Tr}    \makeop{Trd}   \makeop{diag}
\makeop{Res}   \makeop{ind}   \makeop{depth} \makeop{Tr}    \makeop{st}
\makeop{Ad}    \makeop{Int}   \makeop{tr}    \makeop{Sym}   \makeop{can}
\makeop{length}\makeop{SO}    \makeop{torsion} \makeop{GSp} \makeop{Ker}
\makeop{Adm}
\def\makebb#1{\expandafter\def
  \csname bb#1\endcsname{{\mathbb{#1}}}\ignorespaces}
\def\makebf#1{\expandafter\def\csname bf#1\endcsname{{\bf
      #1}}\ignorespaces} 
\def\makegr#1{\expandafter\def
  \csname gr#1\endcsname{{\mathfrak{#1}}}\ignorespaces}
\def\makescr#1{\expandafter\def
  \csname scr#1\endcsname{{\EuScript{#1}}}\ignorespaces}
\def\makecal#1{\expandafter\def\csname cal#1\endcsname{{\mathcal
      #1}}\ignorespaces} 

\def\doLetters#1{#1A #1B #1C #1D #1E #1F #1G #1H #1I #1J #1K #1L #1M
                 #1N #1O #1P #1Q #1R #1S #1T #1U #1V #1W #1X #1Y #1Z}
\def\doletters#1{#1a #1b #1c #1d #1e #1f #1g #1h #1i #1j #1k #1l #1m
                 #1n #1o #1p #1q #1r #1s #1t #1u #1v #1w #1x #1y #1z}
\doLetters\makebb   \doLetters\makecal  \doLetters\makebf
\doLetters\makescr 
\doletters\makebf   \doLetters\makegr   \doletters\makegr
     \def\qed{\qedmark\medbreak}%
\def\qedmark{{\enspace\vrule height 6pt width 5pt depth 1.5pt}}%

\normalsize

\makeop{Bl}

\def\Fq{{\bbF}_q}

\newcommand{\Z}{\mathbb Z}
\newcommand{\Q}{\mathbb Q}

\newcommand{\A}{\mathbb A}    
\newcommand{\F}{\mathbb F}

\newcommand{\isoto}{\stackrel{\sim}{\to}}
\nc{\embed}{\hookrightarrow}

\nc{\ol}{\overline}
\nc{\wt}{\widetilde}
\nc{\opp}{\mathrm{opp}}

\makeop{Ram}
\makeop{Rep}

\begin{document}
\def\wh{\widehat}
\def\DS{{\rm DS}}
\def\Mass{{\rm Mass}}
\def\Mat{{\rm Mat}}
\def\vol{{\rm vol}}

\renewcommand{\thefootnote}{\fnsymbol{footnote}}
\setcounter{footnote}{-1}
\numberwithin{equation}{section}

\title{Mass formula of division algebras over global function fields}
\author{Fu-Tsun Wei and Chia-Fu Yu}
\address{(Wei) Department of Mathematics, National Tsing-Hua
  University, Hsinchu 30013, Taiwan} 
\email{d947205@oz.nthu.edu.tw}

\address{
(Yu) Institute of Mathematics, Academia Sinica and NCTS (Taipei Office)\\
6th Floor, Astronomy Mathematics Building \\
No. 1, Roosevelt Rd. Sec. 4 \\ 
Taipei, Taiwan, 10617} 
\email{chiafu@math.sinica.edu.tw}

\date{\today}
\subjclass[2000]{11R52,11R58}
\keywords{mass formula, global function fields, central division algebras} 

\begin{abstract}
In this paper we give two proofs of the mass formula for definite central 
division algebras over global function fields, due to Denert and Van
Geel. The first proof is based on a calculation of Tamagawa measures. 
The second proof is based on analytic methods, in which we establish
the relationship directly between the mass and the value of the associated 
zeta function at zero. \\ 

\end{abstract} 

\maketitle


\section{Introduction}
\label{sec:01}

Let $K$ be a global function field with constant field $\Fq$. 
Fix a place $\infty$ of $K$, referred as the place at infinity.
Let $A$ be the subring of functions in $K$ regular everywhere
outside $\infty$. Let $B$ be a definite central division algebra of
dimension $r^2$ over $K$; see Section~\ref{sec:02}. 
Let $R$ be a maximal $A$-order in $B$ and let $G'$ be the
multiplicative group of $R$, regarded as a group scheme over
$A$. Denote by $\wh A$ the pro-finite completion of $A$, which is
the maximal open compact topological subring of the ring
$\A_K^\infty$ of finite adeles of $K$. 
The mass associated to the double coset space $G'(K)\backslash
G'(\A_K^\infty)/G'(\wh A)$ is defined as
\begin{equation}
  \label{eq:11}
  \Mass(G',G'(\wh A)):=\sum_{i=1}^h |\Gamma_i|^{-1}, \quad 
  \Gamma_i:=G'(K)\cap c_i G'(\wh A) c_i^{-1},
\end{equation}
where $c_1,\dots, c_h$ are complete representatives for the double
coset space.

In this paper we prove the following result.

\begin{thm}\label{11} We have
\begin{equation}\label{eq:12}
\Mass(G', G'(\wh A))=\frac{\#\Pic(A)}{q-1}\cdot \prod_{i=1}^{r-1}
\zeta_K(-i)\cdot \prod_{v\in S} \lambda_{v}, 
\end{equation}
where $\Pic(A)$ is the Picard group of $A$, 
\[ \zeta_K(s):=\prod_{v} (1-N(v)^{-s})^{-1}\] is the zeta function of
$K$, $S$ is the finite subset of ramified places for $B$ and
\begin{equation}
  \label{eq:13}
 \lambda_v=\prod_{1\le i\le r-1,\  d_v\nmid \,i} (N(v)^i-1), 
\end{equation}
where $d_v$ is the index of the central simple algebra $B_v=B\otimes_K
K_v$. 
\end{thm}

We remark that the mass (\ref{eq:11}) is defined only when the central
simple algebra is definite. For the complementary cases
where the central simple algebra is not definite, one easily shows
that its class number is equal to the class number of $A$
(Corollary~\ref{23}). This is the analogue of the classical theorem 
(due to Eichler) that any central simple algebra over $\Q$ which is
not a definite quaternion algebra has class number one. 

We say that a central simple algebra $B$ over $K$ is of Drinfeld type
if the invariant of $B$ at $\infty$ is $-1/r$ and $B$ is ramified at
one more (finite) place $\grp\subset A$. 
Recall that a Drinfeld $A$-module $\phi$
of rank $r$ over a $\kappa(\grp)$-field $\kappa_1$ is called {\it
  supersingular} if the group of $\bar \kappa_1$-valued points 
of the $\grp$-torsion
subgroup $\phi[\grp]$ is trivial, where $\bar \kappa_1$ denotes an
algebraic 
closure of $\kappa_1$. Let $\Sigma(r,\grp)$ denote the set
of isomorphism classes of supersingular Drinfeld $A$-modules of rank
$r$ over $\ol{\kappa(\grp)}$. The set $\Sigma(r,\grp)$ is in bijection
with the double space 
$ G'(K)\backslash G'(\A_K^\infty)/G'(\wh A)$ associated to the algebra
$B$ of Drinfeld type ramified at $\{\infty,\grp\}$ and each object
$\phi$ in $\Sigma(r,\grp)$ has only finitely many automorphisms.
One associates the geometric mass $\Mass(\Sigma(r,\grp))$ as 
\begin{equation}
  \label{eq:14}
  \Mass(\Sigma(r,\grp)):=\sum_{[\phi]\in \Sigma(r,\grp)}
  |\Aut(\phi)|^{-1}. 
\end{equation}
As an immediate consequence of Theorem~\ref{11} applied to Drinfeld type
division algebras $B$, we obtain the
following geometric mass formula \cite[Theorem
2.1]{yu-yu:ssd}. This is the function field analogue of the
Deuring-Eichler mass formula for supersingular elliptic curves.
\begin{thm}\label{12} We have
  \begin{equation}
    \label{eq:15}
  \Mass(\Sigma(r,\grp)):=\frac{\#\Pic(A)}{q-1}\cdot \prod_{i=1}^{r-1}
\zeta^{\infty,\grp}_K(-i),  
  \end{equation}
where $\zeta^{\infty,\grp}_K(s)=\prod_{v\neq \infty, \grp}
(1-N(v)^{-s})^{-1}$ is the zeta function of $K$ with the local
factors at $\infty$ and $\grp$ removed.   
\end{thm}
Theorem~\ref{12}  was proved by Gekeler when $r=2$ or $K$ is the
rational function field (\cite[Theorem 1, p.~144]{gekeler:quaternion}, 
\cite[Theorem 2.5, p.~321 and 5.1, p.~328]{gekeler:mass}), and by
J.~Yu and the second author \cite{yu-yu:ssd} for both arbitrary $r$
and global function fields $K$. The proof in loc.~cit. 
consists of two parts: The first one deduces
the mass, through manipulating Tamagawa measures, 
as a  product of zeta values up to explicit local indices (the
ratio of the volumes of two local open compact groups at each ramified
place); see also (\ref{eq:38}). Then one uses Gekeler's result of
geometric mass formula for the rational global function field case to 
determine the local index. This argument yields the local index
where the local invariant of $B$ is $\pm 1/r$ for free, however, 
its proof roots in the result of counting supersingular Drinfeld
modules in the Drinfeld moduli scheme modulo the finite prime $\grp$.    

Since central division algebras 
considered in Theorem~\ref{11} may not arise from geometry, 
that is, as endomorphism algebras of certain Drinfeld modules,
the question of determining the mass formula goes beyond the reach of
geometric methods and hence a different approach is needed. 
In this paper we give two proofs of Theorem~\ref{11}.
As a consequence we obtain two different proofs of the 
geometric mass formula (\ref{eq:15}). 
For the first proof we calculate the remaining the ratio 
of local volumes directly. The proof is given in Section~\ref{sec:05};
some basic results in central simple algebras over local fields are
recalled in Section~\ref{sec:04}.

The second proof is analytic. 
In this proof we directly show that the associated
mass is equal to the value of the associated zeta function at zero;
see Subsection~\ref{sec:61}. 






After the present manuscript was completed, the authors learned that
Theorem~\ref{11} was first obtained in Denert and van Geel
\cite{denert-geel:classno88} and that it is also a consequence of deep
results of Gopal Prasad~\cite{prasad:s-volume}. Therefore, the main
result presented in this paper is not new. However, we hope that the
detailed calculations presented in this paper may be helpful to some readers
who wish to know more elementary steps.  



\section{Preliminaries}
\label{sec:02}

\subsection{Notation}
\label{sec:21}

Let $K$ be a global function field with constant field $\F_q$. 
Fix a place $\infty$ of $K$, referred as the place at infinity. All
other places of $K$ are referred as finite places. Let $A$
be the subring of functions in $K$ regular everywhere outside $\infty$. For
each place $v$ of $K$, denote by $K_v$ the completion of $K$
at $v$, and denote by $O_v$ the ring of integers in $K_v$. When $v$ is
finite, the ring $O_v$ equals the completion $A_v$ of $A$ at $v$.  
We also write  $\kappa(v)$ for the residue field $O_v/\pi_v$ at $v$, 
where $\pi_v$ is a uniformizer of $O_v$, and put
$N(v):=\# \kappa(v)$. 


For any $A$-module or $K$-module $M$, we write $M_v$
for $M\otimes_A A_v$ if $v$ is finite, and $M_v$ for $M\otimes_K K_v$ for any
place $v$. Let $\A^\infty_K$ denote 
the ring of finite adeles of $K$
(with respect to $\infty$) and put
\[ \wh A:=\prod_{v:{\rm finite}} A_v,  \]
the pro-finite completion of $A$.  

For a linear algebraic group $G$ over
  $K$ and an open compact subgroup $U$ of $G(\A^\infty_K)$, 
denote by $\DS (G,U)$ the double coset space $G(K)\backslash
  G(\A_K^\infty)/U$. If the arithmetic subgroup $G(K)\cap U$ is finite,
  or equivalently that any ($\infty$-)arithmetic subgroup $\Gamma$ of $G(K)$
  (i.e. $\Gamma$ is a subgroup commensurable to $G(K)\cap U$
  ) is finite, define
\[ \Mass(G,U):=\sum_{i=1}^h |\Gamma_i|^{-1}, \quad \Gamma_i:=G(K)\cap
  c_i U c_i^{-1}, \]
where $c_1,\dots,c_h$ are complete representatives for $\DS(G,U)$. 
It is easy to show that
$\Mass(G,U)$ does not depend on the choice of representatives $c_i$.  

\subsection{Class numbers of indefinite central simple algebras}
\label{sec:22}

Let $B$ be a central simple algebra over $K$. 
An {\it $A$-order} in $B$ is an $A$-subring
of $B$ which is finite as an $A$-module and spans $B$ over $K$. 
An $A$-order in $B$ is called {\it maximal} if
it is not properly contained in another $A$-order in $B$. 
Let $\Lambda$ be a maximal $A$-order in $B$. By a right 
fractional ideal of $\Lambda$ we mean a non-zero finite
right $\Lambda$-submodule $I$ in $B$; $I$ is called {\it full} if it
spans $B$ over $K$. When $B$ is a division algebra, any fractional
ideal of $\Lambda$ is full. Let $\calL$ be the set of all
full right fractional ideals of $\Lambda$ in $B$. 
Two right fractional ideals $I$
and $I'$ are said to be {\it locally equivalent at a finite place $v$} 
if $I'_v=g_v I_v$ for some element $g_v\in B_v^\times$; 
they are said to be {\it globally equivalent} if there is an element
$g\in B^\times$ such that $I'=g I$. This is 
equivalent to that $I_v$ and $I'_v$ (resp. $I$ and $I'$) are
isomorphic as $\Lambda_v$-modules (resp. as $\Lambda$-modules).
 
Since $\Lambda_v$ is a maximal $A_v$-order, any one-sided ideal of 
$\Lambda_v$ is principal (\cite[Theorem 18.7, p.~179]{reiner:mo}). 
It follows that
any two full right fractional ideals are
locally equivalent everywhere, that is, the set $\calL$ of ideals forms a
single genus. Let $\calL/\sim$ denote the set of global equivalence 
classes of right ideals in $\calL$. Let $G'$ be the group scheme over $A$
associated to the multiplicative group of $\Lambda$. 
For each commutative $A$-algebra $L$, the group of $L$-valued 
points of $G'$ is 
\[ G'(L)=(\Lambda\otimes_A L)^\times. \]
The above argument establishes the following well-known basic fact:

\begin{lemma}\label{21}
  There is a natural bijection
\[ \varphi: G'(K)\backslash G'(\A_K^\infty)/G'(\wh A) \to \calL/\sim \]
which maps the identity class to the trivial class $[\Lambda]$.
\end{lemma}

The cardinality of $\calL/\sim$ is
independent of the choice of the maximal order $\Lambda$; this follows
from the basic fact that any two maximal orders are locally
conjugate. The number $\# \calL/\sim$ 
is called the class number of $B$ (relative to $\infty)$, which we
denote $h^\infty(B)$ or simply $h(B)$ as the place $\infty$ has been fixed.  
We shall call $B$ definite (at $\infty$) if $B_\infty:=B\otimes_K
K_\infty$ is a division algebra, and $B$ indefinite (at $\infty$)
otherwise. 

\begin{lemma}\label{22}
Assume that $B$ is indefinite.
Let $U$ be an open compact subgroup of $G'(\A_K^\infty)$. Then 
the reduced norm map $N_{B/K}$ induces a bijection of double coset spaces
\[ N_{B/K}: G'(K)\backslash G'(\A_K^\infty)/U \simeq
  K^\times \backslash\A_K^{\infty,\times}/N_{B/K}(U). \]
\end{lemma}
\begin{proof}
  This follows from the strong approximation; we provide the proof for
  the reader's convenience. We may assume that $B\neq K$. 
  Clearly the induced map is surjective. We show the
  injectivity. Let $[a]$ be an element in the target space. Fix a section
  $s:\A_{K}^{\infty,\times} \to G'(\A_{K}^\infty)$ of the map
  $N_{B/K}$. Then 
  the inverse image $T_{[a]}$ of the class $[a]$ consists of elements
  $_{G'(K)} [g s(a)]_U$ for all $g\in G'_1(\A_K^{\infty})$, where
  $G'_1\subset G'$ is the reduced norm one algebraic subgroup. The
  surjective map $g\mapsto\,  _{G'(K)}[g s(a)]_U$ induces a
  surjective map 
\[  \alpha: G'_1(K)\backslash G'_1(\A_{K}^{\infty})/U' \to T_{[a]}, \]
where $U':=s(a) U s(a)^{-1} \cap G'_1(\A_{K}^{\infty})$. 
Since the group $G'_1$ is semi-simple, simply connected and
$G'_1(K_\infty)$  is not compact, the strong approximation 
holds for the algebraic group $G'_1$.  
Therefore, $T_{[a]}$ consists of a single element and this proves 
the lemma. \qed   
\end{proof}

\begin{cor}\label{23} Assume that $B$ is indefinite.
\begin{enumerate}
\item We have $h(B)=\# \Pic (A)=:h(A)$, where $\Pic (A)$ is the Picard
      group of $A$.

\item If $A$ is a principal ideal domain, then any full one-sided
  ideal of $\Lambda$ is principal. 
\end{enumerate}
\end{cor}
\begin{proof}
  These easily follow from Lemmas~\ref{21} and \ref{22}. \qed
\end{proof}

\begin{lemma}\label{24}
  Notation as above. The algebra $B$ is definite
  if and only if any $\infty$-arithmetic subgroup $\Gamma$ of $G'(K)$
  is finite. 
\end{lemma}
\begin{proof} This is clear. 

\end{proof}\ 

From discussion above, the mass $\Mass(G',U)$ is defined only when the
central simple algebra $B$ is definite. When $B$ is indefinite, the
class number $h(B)$ is equal to $h(A)$. One can calculate the class
number $h(A)$ of $A$ by the following formula 
\cite[p.~143, (1.5)]{gekeler:quaternion}:
\begin{equation}
  \label{eq:21}
  h(A)=\deg \infty P(1),
\end{equation}
where $\deg \infty$ is the degree of $\infty$, and $P(T)\in \Z[T]$ is
the polynomial so that 
\[ \zeta_K(s)=\frac{P(q^{-s})}{(1-q^{-s})(1-q^{1-s})}. \]
In the sequel we shall only consider the case 
where $B$ is a definite central division algebra over $K$.
  
\section{Proof of Theorem~\ref{11}}
\label{sec:03}

Keep the notation as in Section~\ref{sec:01}; $B$ and $R$ are as
before. 
Let $B_0$ be the matrix algebra $\Mat_r(K)$ and let $R_0:=\Mat_r(A)$
be the standard maximal $A$-order in $B_0$.
Let $G$ and $G'$ be the group schemes over $A$ 
associated to the multiplicative
groups of $R_0$ and $R$, respectively. 
Let $G_1$ (resp. $G'_1$) denote the reduced norm one subgroup schemes 
of $G$ (resp. $G'$).

First we have 
\begin{equation}
  \label{eq:31}
  \Mass (G'(\wh A)):=\Mass(G',G'(\wh A))=\frac{\vol(G'(K)\backslash
  G'(\A^\infty_K))}{\vol(G'(\wh A))}, 
\end{equation}
for any Haar measure $dg'$ on $G'(\A^\infty_K)$. A simple computation
(cf. \cite{yu-yu:ssd}) shows that
\begin{equation}\label{eq:35}
  \begin{split}
   \Mass(G'(\wh A))& =\frac{\#\Pic(A)}{q-1} 
   \cdot \tau(G'_1)\cdot \omega'_\A(P')^{-1},\quad
    \\
   & =\frac{\#\Pic(A)}{q-1} 
   \cdot \omega'_\A(P')^{-1} \quad (\tau(G'_1)=1, \text{Weil's Theorem
   \cite{weil:adele}})   
  \end{split}
\end{equation}
where $P':=\prod_v P'_v \mbox{ with } P'_v:=G'_1(O_v)$,
$\omega'_{\A}$ is the Tamagawa measure on $G'_1(\A_K)$ and
$\tau(G'_1)$ is the Tamagawa number of $G'_1$.  

Let $\omega$ be an invariant $K$-rational differential 
form of top degree on the group $G_1$, 
and let $\omega'$ be the pull back of $\omega$ via an
inner twist 
$\alpha:G'_1\isoto G_1$ (over a finite extension of $K$). 
They give rise to the Tamagawa measures
$\omega_\A$ and $\omega'_\A$ on $G_1$ and $G'_1$, respectively. Then
\begin{equation}\label{eq:36}
 \omega'_\A(P')=\omega_\A(P) \cdot \prod_{v\in S}
 \frac{\omega'_{v}(P'_v)}
 {\omega_{v}(P_v)},
\end{equation}
where $P=\prod_v P_v$, $P_v:=G_1(O_v)$ and 
$S$ is the finite set of ramified
places for $B$. From the well-known fact that
$\omega_\A(P)^{-1}=\prod_{i=1}^{r-1} \zeta_K(-i)$, we get 
\begin{equation}\label{eq:38}
\Mass(G'(\wh A))=\frac{\#\Pic(A)}{q-1}\cdot \prod_{i=1}^{r-1}
  \zeta_K(-i)\cdot \prod_{v\in S} \lambda_{v}, \quad 
  \lambda_v:= \frac{\omega_{v}(P_v)}{\omega'_{v}(P'_v)} .
\end{equation}

\begin{prop}\label{31}
  Suppose that $B_v\simeq \Mat_{m_v}(\Delta_v)$, where $\Delta_v$ is the
  division part of $B_v$, and let $d_v$ be the index of
  $\Delta_v$. Then 
  \begin{equation}
    \label{eq:310}
  \lambda_v=\prod_{1\le i\le r-1,\  d_v\nmid \,i} (N(v)^i-1). 
  \end{equation}
\end{prop}
The proof of Proposition~\ref{31} will be given in
Section~\ref{sec:05}. By (\ref{eq:38}) and  Proposition~\ref{31},
Theorem~\ref{11} is proved.

\section{Division algebras over local fields}
\label{sec:04}

In this section we make preparation on  
central division algebras over non-Archimedean local fields. This
will be used in the next section. Let $K_v$, $O_v$, $\pi_v$,
$\kappa(v)$, $N(v)$ be as before. 

\subsection{Maximal orders}
\label{sec:41}
Let $\Delta$ 
be a central division algebra of dimension $d^2$ over $K_v$. 
Let $L$ be the unramified field extension of $K_v$ of degree $d$, and
$O_L$ its ring of integers. Let $\sigma$ be the (arithmetic) 
Frobenius automorphism of $L$ over $K_v$. 
Suppose $\inv(\Delta)=b/d$, where $b$ is a positive integer with $(b,d)=1$
and $b<d$. We use the normalization of invariant of $\Delta$ in
Pierce \cite{pierce}; see p.~338 and p.~277. We can write 
\begin{equation}
  \label{eq:41}
  \Delta=L[\Pi'], \quad (\Pi')^d=\pi_v^b, \quad (\Pi')^{-1} c
  \Pi'=\sigma(c), \quad \forall\, c\in L.
\end{equation}
Note that the normalization in Reiner \cite{reiner:mo} is different;
the invariant of $\Delta$ in (\ref{eq:41}) is defined to be $-b/d$
there; see \cite[(31.7), p.~266 and p.~264]{reiner:mo}. 

Choose integers $m$ and $m'$ such that $bm+dm'=1$. We may take $1\le
m\le d$. Put $\Pi:=(\Pi')^m \pi_v^{m'}$. It is easy to check that
\begin{equation}
  \label{eq:42}
  \Pi^d=(\pi_v)^{bm} \pi_v^{dm'}=\pi_v, \quad \Pi^{-1} c
  \Pi=\sigma^m(c),\quad \forall\, c\in L.
\end{equation}
Put $\tau:=\sigma^m$; we have $\Gal(L/K_v)=<\tau>$. The subring
\[ O_\Delta:=O_L[\Pi]\subset \Delta \] 
is the unique maximal order; see \cite[Theorem 13.3,
p.~140 and p.~146]{reiner:mo}. 

We regard $\Delta$ as a right vector space over $L$, with basis 
$1,\Pi, \dots, \Pi^{d-1}$. The left translation of $\Delta$ on 
$\Delta$ gives an embedding 
\[ \Phi: \Delta\to \Mat_d(L) \]
as $K_v$-algebras. From the relation $a_0 \Pi^i=\Pi^i \tau^i (a_0)$ for 
$a_0\in L$, we have 
\begin{equation}
  \label{eq:43}
  \Phi(a_0)=
  \begin{pmatrix}
    a_0 & 0 & \dots & 0 \\
     0  & \tau(a_0) & \dots & 0 \\
    \vdots & \vdots & \ddots  &  0 \\
    0  & 0 &  & \tau^{d-1}(a_0)
  \end{pmatrix}, 
  \quad \Phi(\Pi)=
  \begin{pmatrix}
    0 & 0& \cdots & 0 & \pi_v \\
    1 & 0 & & 0 & 0 \\
    0 & 1 & \ddots & & \vdots \\
    \vdots & \vdots & \ddots & \ddots & \vdots \\
    0 & 0 & \cdots & 1 & 0 \\    
  \end{pmatrix}.
\end{equation}
For example, when $d=3$, we have
\begin{equation}
  \label{eq:44}
  \Phi \left ( a_0+\Pi a_1+\Pi^2 a_2\right )=
  \begin{pmatrix}
    a_0 & \pi_v \tau(a_2) & \pi_v \tau^2(a_1) \\
    a_1 & \tau(a_0)       & \pi_v \tau^2(a_2) \\
    a_2 & \tau(a_1) & \tau^2(a_0) \\
  \end{pmatrix}.
\end{equation}
The map $\Phi:\Delta\to \Mat_d(L)$ induces an isomorphism 
\begin{equation}
  \label{eq:45}
  \Phi_L: \Delta\otimes_{K_v} L \to \Mat_d(L), \quad x\otimes a
  \mapsto \Phi(x)a.
\end{equation}
Let $Iw\subset \Mat_d (O_L)$ be the hereditary order 
\begin{equation}
  \label{eq:46}
  Iw:=\left \{(a_{ij})\in \Mat_d(O_L)\, |\, a_{ij}\in \pi_v O_L,\  
  \forall\, i<j\,\right \}. 
\end{equation}
It is not hard to see that 
\begin{equation}
  \label{eq:47}
  O_\Delta=\left \{a\in \Delta\, |\, \Phi(a)\in Iw\, \right \}, 
  \quad \text{and}\quad \Phi_L(O_\Delta\otimes_{O_v} O_L)=Iw. 
\end{equation}
In other words, the map $\Phi$ is an optimal embedding (also called
a maximal embedding) of $O_\Delta$ into the order
$Iw\subset \Mat_d(L)$.

\subsection{Haar measures and the base change formula}
\label{sec:42}
Let $\{e'_{ij}\}_{1\le i,j\le d}$ be a $K_v$-basis for the vector space
  $\Delta$. For 
any element $x'=\sum_{i,j} x'_{ij} e'_{ij}$ in $\Delta$, write
$x'=(x'_{ij})$ and $x'_{ij}$s are global linear coordinates for
$\Delta$, regarded as a commutative algebraic group over $K_v$. 
The invariant differential form $dx'=\prod_{i,j} dx'_{ij}$
of top degree naturally gives rise to an additive Haar measure on $\Delta$, 
which we also denote $dx'$, by setting 
\[ \vol(B(1), dx')=1, \]
where $B(1):=\{(x'_{ij})\, |\, x'_{ij}\in O_v, \ \forall\, i,j\, \}$
is the unit ball. Let 
\[ d^\times x':=\frac{d x'}{|N_{\Delta/K_v}(x')|_v^{d}}, \]
be the induced Haar measure on $\Delta^\times$, 
where $N_{\Delta/K_v}$ is the reduced norm map and
$|\pi_v|_v=N(v)^{-1}$. Regarding $G'=\Delta^\times$ as an algebraic group
over $K_v$, $d^\times x'$ is also an invariant different form on $G'$
of top degree. 

The differential form $dx'$ is a $K_v$-rational differential form on
$\Delta\otimes L$, regarded as a commutative algebraic group over $L$. 
The induced Haar measure on $\Delta\otimes L$ will
be denoted by $dx'\otimes L$. 

\begin{prop}\label{41} \
\begin{enumerate}
\item For any full $O_v$-lattice $M$ in $\Delta$, we have the base change
  formula
  \begin{equation}
    \label{eq:48}
    \vol(M, dx')^d =\vol(M\otimes_{O_v} O_L, dx'\otimes L). 
  \end{equation}
\item We have 
  \begin{equation}
    \label{eq:49}
     \vol(O_\Delta, dx') =N(v)^{-d(d-1)/2}
     \left [\vol(\Phi_L^{-1}(\Mat_{d}(O_L)), dx'\otimes L)\right ]^{1/d}. 
  \end{equation}
\end{enumerate}
\end{prop}
\begin{proof}
  (1) This is clear. (2) It follows from (1) and (\ref{eq:47}) that
  \begin{equation*}
    \begin{split}
     \vol(O_\Delta, dx')&=
  \left [\vol (\Phi_L^{-1}(Iw), dx'\otimes L)\right ]^{1/d}\\ 
  &=\left [ \Mat_d(O_L):Iw\right ]^{-1/d}
  \left [\vol(\Phi_L^{-1}(\Mat_{d}(O_L)), dx'\otimes L)\right ]^{1/d}. 
    \end{split}
  \end{equation*}
Then (\ref{eq:49}) follows from 
\[ \left [\Mat_d(O_L):Iw\right ]=N(v)^{d^2(d-1)/2}. \]  \qed
\end{proof}

\section{Computation of local indices}
\label{sec:05}
In this section we shall give a proof of Proposition~\ref{31}. 
Suppose $B_v\simeq \Mat_{m_v}(\Delta_v)$, where $\Delta_v$ is a
central division algebra of dimension $d_v^2$ over $K_v$. 
We have $r=m_v d_v$.
  
Choose the standard coordinates $x_{ij}$ for $\Mat_r(K_v)$ and form an
invariant differential form $dx:=\prod_{i,j} dx_{ij}$ of top degree on
the commutative algebraic group 
$\Mat_r$ over $K_v$. Let $L$ be the unramified extension of $K_v$ of
degree $d_v$. The $L$-algebra isomorphism $\Phi_L: B_v\otimes_{K_v} L
\to \Mat_r(L)$ constructed in Section~\ref{sec:04} gives an 
isomorphism $\alpha: B_v \to \Mat_r$ of ring schemes over $L$, and
defines also an isomorphism $\alpha: G' \to G$ over $L$. The
pull-back differential form $\alpha^* dx$ is $K_v$-rational and 
there is an invariant differential form $dx'$ of top degree on $B_v$ such
that $dx'\otimes L=\alpha^* dx$. Then the invariant differential forms
\[ dg:=dx/|\det(x)|_v^r, \quad dg':=dx'/|N_{B_v/K_v}(x')|_v^r \]
induce the Haar measures on $G'(K_v)$ and $G(K_v)$ which are
transferred to
each other via the map $\alpha$. 
Choose a Haar measure $dt$ of $K_v^\times$; it defines Haar measures
$dg_1$ on $G_1(K_v)$ and $dg'_1$ on $G'_1(K_v)$ such that 
$dg=dg_1 dt$ and $dg'=dg_1' dt$. Also $dg_1'$ is the transfer of $dg_1$. 
It follows that
\[ \lambda_v=\frac{\vol(G_1(O_v), dg_1)}{\vol(G'_1(O_v), dg'_1)}
     =\frac{\vol(G(O_v), dg)}{\vol(G'(O_v), dg')}.\]
We shall calculate the volumes $\vol(G(O_v), dg)$ and $\vol(G'(O_v),
dg')$. 
From our choice of the Haar measure, we have $\vol(\Mat_r(O_v),
dx)=1$. Therefore, 
\begin{equation}
  \label{eq:51}
  \vol(G(O_v))=\int_{G(O_v)} dg=\int_{G(O_v)} dx=\vol(\Mat_r(O_v),dx)
  \frac{\# \GL_r(\kappa(v))}{\# \Mat_r(\kappa(v))}. 
\end{equation}
It is known that $\# \GL_r(\kappa(v))=N(v)^{r(r-1)/2} 
\prod_{i=1}^r (N(v)^i-1)$, and we get
\begin{equation}
  \label{eq:52}
  \vol(G(O_v))=\frac{\# \GL_r(\kappa(v))}{\# \Mat_r(\kappa(v))}=
  \frac{\prod_{i=1}^r (N(v)^i-1)}{N(v)^{r(r+1)/2}}.
\end{equation}
On the other hand, we have
\begin{equation}
  \label{eq:53}
  \vol(G'(O_v))=\vol(\Mat_{m_v}(O_{\Delta_v}),dx') 
  \frac{\# \GL_{m_v}(O_{\Delta_v}/\Pi_v)}{\#
  \Mat_{m_v}(O_{\Delta_v}/\Pi_v)}. 
\end{equation}
It follows from Proposition~\ref{41} that
\[ \vol(\Mat_{m_v}(O_{\Delta_v}),dx')=N(v)^{-m_v^2 d_v(d_v-1)/2}. \]
Similar to (\ref{eq:52}), we have
\[ \frac{\# \GL_{m_v}(O_{\Delta_v}/\Pi_v)}{\#
  \Mat_{m_v}(O_{\Delta_v}/\Pi_v)}=\frac{\prod_{i=1}^{m_v} (N(v)^{i
  d_v}-1)}{N(v)^{d_v m_v(m_v+1)/2}}. \]
Therefore, we get
\begin{equation}
  \label{eq:54}
  \vol(G'(O_v))=\frac{\prod_{i=1}^{m_v} (N(v)^{i d_v}-1)}{ N(v)^{m_v d_v(m_v
d_v+1)/2}}. 
\end{equation} 
From (\ref{eq:52}) and (\ref{eq:54}), we get
\[ \lambda_v=\prod_{1\le i\le r-1, \atop d_v\nmid i} (N(v)^i-1). \]
This proves Proposition~\ref{31}. 

\section{Alternative approach via zeta functions}\label{sec:06}

This section is an analytic proof for Theorem 1.1. Keep the notation
as in Section~\ref{sec:01} and Subsection~\ref{sec:21}. Particularly 
we have chosen the definite central division algebra $B$ over $K$ of
dimension $r^2$ and a maximal $A$-order $R$ in $B$. 
Fix complete representatives $c_1,\ldots,c_h$ for the double coset space
$G'(K)\backslash G'(\mathbb{A}_K^{\infty})/G'(\widehat{A})$
where $G'$ is the group scheme over $A$ defined as before. 
For $1\leq i \leq h$, let 
$$I_i:= B \cap c_i\widehat{R}\quad \text{ and \ } R_i:= B \cap
c_i\widehat{R}c_i^{-1},$$ 
where $\widehat{R}:= R \otimes_A \widehat{A}$, the pro-finite
completion of $R$. 
Then $I_1,\ldots,I_h$ are complete representatives of right ideal
classes of $R$, 
and $R_i$ is the left order of $I_i$ for each $i$.
The \it inverse \rm of $I_i$ is
$$I_i^{-1} := B \cap \widehat{R}c_i^{-1}.$$
One has $I_i^{-1}\cdot I_i = R$ and $I_i \cdot I_i^{-1} = R_i$.
The units group $R_i^{\times}$ of $R_i$ is equal to $G'(K) \cap
c_iG'(\widehat{A})c_i^{-1}$ and 
$$\text{Mass}(G'(\widehat{A}))=\sum_{i=1}^h\frac{1}{\#(R_i^{\times})}.$$

\subsection{Partial zeta functions}\label{sec:61}

For $1\leq i \leq h$, define the \it partial zeta function \rm
$$\zeta_i(s):= \sum_{I \sim I_i, I \subset R} \frac{1}{|N_{B/K}(I)|^s}.$$
Here $N_{B/K}(I)$ is the fractional ideal of $A$ generated by the
reduced norm $N_{B/K}(\alpha)$ of elements $\alpha$ in $I$,  
and $|\mathfrak{m}|$ is the cardinality of $A/\mathfrak{m}$ for
a non-zero ideal $\mathfrak{m} \subset A$. 
From the definition of $\zeta_i(s)$ one has
$$\zeta_i(s) = \sum_{\text{ideals }\mathfrak{m} \subset A}
\frac{b_{i}(\mathfrak{m})}{|\mathfrak{m}|^s},$$ 
where $b_i(\mathfrak{m}) := \#\{I \subset R: I \sim I_i \text{ with }
N_{B/K}(I) = \mathfrak{m}\}$. 

\begin{prop}\label{61}
The function $\zeta_i(s)$ converges absolutely for $\text{Re}(s)>r$
for all $i=1,\dots, h$
and has a meromorphic continuation to the whole complex plane 
with a simple pole at $s =r$.
Moreover, one has
$$\zeta_i(0) = -\frac{1}{\#(R_i^{\times})}.$$
\end{prop}

\begin{proof}
Given a right ideal $I \subset R$ with $I \sim I_i$ and $N_{B/K}(I) =
\frak{m}$. 
There exists a unique $\alpha \in I_i^{-1}$, up to multiplying
elements in $R_i^{\times}$ from the right,  
such that $I = \alpha I_i$ and $N_{B/K}(\alpha)N_{B/K}(I_i) = \frak{m}$.
Hence
$$\#(R_i^{\times}) \cdot b_i(\mathfrak{m}) = \#\{\alpha \in I_i^{-1}:
N_{B/K}(\alpha)N_{B/K}(I_i) = \mathfrak{m} \}.$$ 

Let $\deg$ be the usual degree map on the divisor group
$\text{Div}(K)$ of $K$, i.e.\ 
$\deg v = [\kappa(v):\mathbb{F}_q]$ for any place $v$ of $K$.
Choose the valuation $v_{\infty}$ on $K_{\infty}$ normalized so that
for $a \in K_{\infty}$ 
$$v_{\infty}(a):= \deg \infty \cdot \text{ord}_{\infty}(a),$$
and the valuation $V_{\infty}$ on $B_{\infty}:=\ B\otimes_K K_{\infty}$
(remembering $B_\infty$ is a division algebra) with
$$V_{\infty}(\alpha):= v_{\infty}(N_{B/K}(\alpha))$$
for $\alpha \in B_{\infty}$.
Identifying fractional ideals of $A$ with divisors of $K$ supported
outside $\infty$, 
we get
$$\#(R_i^{\times}) \cdot \zeta_i(s) = \sum_{\ell=0}^{\infty} a_i(\ell)
q^{-\ell s},$$ 
where
\begin{eqnarray}
a_i(\ell) &:=& \#(R_i^{\times}) \cdot \sum_{\text{ideals }\mathfrak{m}
  \subset A, \atop \deg \mathfrak{m} = \ell} b_i(\mathfrak{m})  
\nonumber \\
        & =& \#\{\alpha \in I_i^{-1}: V_{\infty}(\alpha) = -\ell +\deg
  N_{B/K}(I_i)\}. \nonumber 
\end{eqnarray}
Since $V_{\infty}(\alpha) \equiv 0 \bmod \text{deg} \infty$ for all
$\alpha \in B^{\times}$, 
$$a_i(\ell) = 0 \text{ if } -\ell+\deg N_{B/K}(I_i) \not\equiv 0 \bmod
\deg\infty.$$ 

Set $\mathcal{O}_{B_{\infty}}:= \{w \in B_{\infty}: V_{\infty}(w)\geq 0\}$,
the maximal compact subring of $B_{\infty}$.
Fix an element $\Pi_{\infty}$ in $\mathcal{O}_{B_{\infty}}$ with
$V_{\infty}(\Pi_{\infty}) = \deg \infty$. 
From the Riemann-Roch theorem for the function field $K$
\cite[Chapter VI]{weil:bnt} one can deduce that
for a sufficient large integer $n_0$
$$B_{\infty} = I_i^{-1} + \Pi_{\infty}^{-n_0}\mathcal{O}_{B_{\infty}}.$$
Note that for any integer $\mu$ in $\mathbb{Z}$ one has 
$$\#\big(\Pi_{\infty}^{-\mu-1}\mathcal{O}_{B_{\infty}}
\big/\Pi_{\infty}^{-\mu}\mathcal{O}_{B_{\infty}}\big)= N(\infty)^{r}.$$ 
Therefore when $\mu \geq n_0$,
$$ \#\bigg( \frac{
\{\alpha \in I_i^{-1}\,;\,V_{\infty}(\alpha) \geq -(\mu+1)\deg \infty\}}
{\{\alpha \in I_i^{-1}\,;\,V_{\infty}(\alpha) \geq -\mu \deg \infty\}} 
\bigg) = N(\infty)^r.$$ 

Choose $n_0$ large enough so that $n_0 \deg \infty \geq -\deg
N_{B/K}(I_i)$ and set 
$$\ell_i:= n_0 \deg \infty + \deg N_{B/K}(I_i) $$
and  
$$ C_i:= \#\{\alpha \in I_i^{-1}:V_{\infty}(\alpha)\geq -n_0 \deg
\infty\}.$$ 
From the definition of $a_i$ one has that
for any positive integer $\mu$
$$a_i(\ell_i+\mu \deg \infty)= \#\{\alpha \in
I_i^{-1}:V_{\infty}(\alpha) = -(n_0+\mu)\deg\infty\}.$$ 
By induction on $\mu$ we obtain
$$a_i(\ell_i+\mu \deg \infty) =
(N(\infty)^{r}-1)N(\infty)^{r(\mu-1)}C_i.$$ 
Therefore
\begin{eqnarray}
\#(R_i^{\times})\cdot \zeta_i(s) &=& \sum_{\ell=0}^{\ell_i} a_i(\ell)
q^{-\ell s} +  
q^{-\ell_is}\sum_{\mu=1}^{\infty}a_i(\ell_i+\mu \deg \infty)
N(\infty)^{-\mu s} \nonumber  \\ 
                   &=& \sum_{\ell=0}^{\ell_i} a_i(\ell) q^{-\ell s} +
                   C_i \cdot q^{-\ell_is}\cdot
                   \frac{N(\infty)^{r}-1}{N(\infty)^{r}} 
                   \cdot \sum_{\mu = 1}^{\infty}
                   N(\infty)^{-\mu(s-r)}. \nonumber 
\end{eqnarray}
This shows that $\zeta_i(s)$ converges absolutely for $\text{Re}(s)>r$
with a simple pole at $s=r$, 
and the meromorphic continuation of $\zeta_i$ is:
$$\#(R_i^{\times}) \cdot \zeta_i(s)=
\sum_{\ell=0}^{\ell_i}a_i(\ell)q^{-\ell s}  
+ C_i\cdot q^{-\ell_is} \cdot
\frac{N(\infty)^{r}-1}{N(\infty)^{r}}\cdot
\frac{N(\infty)^{(r-s)}}{1-N(\infty)^{(r-s)}}.$$ 

From the definition of $C_i$ one has $C_i = 1 +
\sum\limits_{\ell=0}^{\ell_i}a_i(\ell)$. 
Hence 
$$\zeta_i(0) = -\frac{1}{\#(R_i^{\times})}.$$ \qed

\end{proof}

\subsection{Mass formula}

Define the \it zeta function for the maximal order \rm $R$:
$$\zeta_{R}(s):= \sum_{ \text{right ideals } I \subset R}
\frac{1}{|N_{B/K}(I)|^s} = \sum_{i=1}^h \zeta_i(s).$$ 
Then $\zeta_{R}(s)$ also has meromorphic continuation and by
Proposition~\ref{61} we have 
$$\zeta_{R}(0) = -\text{Mass}(G'(\widehat{A})).$$

Recall that for each place $v$ of $K$, $B_v:=B \otimes_K K_v$ is
isomorphic to $\text{Mat}_{m_v}(\Delta_v)$, 
where $\Delta_v$ is a central division algebra over $K_v$ with
$\text{dim}_{K_v}\Delta_v = d_v^2$ 
and $m_vd_v = r$.
Then $\zeta_R(s)$ can be expressed by the Dedekind zeta function
$\zeta_K$ of $K$ in the following: 

\begin{thm}\label{62} We have
$$\zeta_{R}(s) = (1-N(\infty)^{-s}) \zeta_K(s) \cdot
\prod_{i=1}^{r-1}\zeta_K(s-i) \cdot  
\prod_{v \in S} \left(\prod_{1\leq i \leq r-1, \atop d_v \nmid
    i}(1-N(v)^{i-s})\right),$$ 
where $S$ is the finite set of ramified places for $B$.
\end{thm}

Let $\zeta_A(s):= (1-N(\infty)^{-s})\zeta_K(s)$, the zeta function for
$A$. Then  
$$\zeta_A(0) = -\frac{\#\text{Pic}(A)}{q-1}$$
where $\text{Pic}(A)$ is the ideal class group of $A$.
Therefore the above theorem tells us

\begin{cor}
\text{\rm (Mass formula)} We have
\begin{eqnarray}
\text{\rm Mass}(G'(\widehat{A})) &=&
\frac{\#\text{\rm Pic}(A)}{q-1} \cdot \prod_{i=1}^{r-1}\zeta_K(-i) \cdot 
\prod_{v \in S} \left(\prod_{1\leq i \leq r-1, \atop d_v \nmid
    i}(1-N(v)^{i})\right) \nonumber \\ 
 &=& \frac{\#\text{\rm Pic}(A)}{q-1} \cdot
 \prod_{i=1}^{r-1}\zeta_K(-i) \cdot  
\prod_{v \in S} \left(\prod_{1\leq i \leq r-1, \atop d_v \nmid
    i}(N(v)^{i}-1)\right). \nonumber 
\end{eqnarray}
\end{cor}

\begin{proof}
The first equality just follows from Theorem~\ref{62}.
Now, for each place $v\in S$,
suppose $\text{inv}(\Delta_v) = b_v/d_v$.
Since $\Delta_v$ is a division ring, the integers $b_v$ and $d_v$ 
are relatively prime. It is well-known that 
\begin{equation}
  \label{eq:61}
  \sum_{v \in S} b_v/d_v \equiv 0 \quad (\bmod\ \mathbb{Z}).
\end{equation}
It follows that 
\begin{equation}
  \label{eq:62}
\sum\limits_{v \in S} r-m_v \equiv 0 \quad (\bmod\ 2).  
\end{equation}
Indeed, if $r$ is odd,
then each term $r-m_v$ is even. Suppose $r$ is even. Let $S_1\subset
S$ be the subset consisting of places $v$ such that $m_v$ is odd. For
each $v\in S_1$, the integer $d_v$ is even and hence $b_v$ is
odd. Since $r$ is even it follows from (\ref{eq:61}) that 
\[  \sum\limits_{v \in S_1} m_v b_v \equiv 0 \quad (\bmod \ 2), \]  
and hence $|S_1|$ is even. 

The second equality follows from (\ref{eq:62}). \qed
\end{proof}

\noindent \text{\it Proof of Theorem~\ref{62}.}
Write $\zeta_{R}(s)$ as $$\sum_{\text{ideals }\mathfrak{m} \subset A}
\frac{b(\mathfrak{m})}{|\mathfrak{m}|^s}$$ 
where  $b(\mathfrak{m}):= \sum_{i=1}^h b_{i}(\mathfrak{m}) =
\#\{\text{right ideals } I \subset R: N_{B/K}(I) = \mathfrak{m}\}$. 
Recall the following bijection
$$\begin{tabular}{ccc}
$G'(\mathbb{A}_K^{\infty})/G'(\widehat{A})$ & $\cong$ & $\{\text{right
  fractional ideals of $R$}\}$ \\ 
 $cG'(\widehat{A})$ & $\longmapsto$ & $B \cap c\widehat{R}$.
\end{tabular}$$
The counting number $b(\frak{m})$ is equal to the number of cosets
$cG'(\widehat{A})$ such that 
$c \in \widehat{R}$ and the coset
$K \cap N_{B/K}(c)\widehat{A} = \frak{m}$. 
Write $c$ as the form $(c_v)_{v \neq \infty}$ where $c_v \in G'(K_v)$.
Then for each finite place $v$ of $K$,
$$N_{B/K}(c_v)\cdot O_v = \frak{m}_v \cdot O_v$$
where $\frak{m}_v \subset A$ is the $v$-component of $\frak{m}$. 
Therefore $$b(\mathfrak{m}) = \prod_{v \neq \infty} b(\mathfrak{m}_v).$$
Let $\frak{p}_v$ denote the ideal of $A$ corresponding to 
the finite place $v$.
Then 
$$\zeta_{B}(s) = \prod_{v \neq \infty} 
\left(\sum_{\ell=0}^{\infty}\frac{b(\mathfrak{p}_v^{\ell})}{N(v)^{\ell
      s}}\right).$$ 

Let $O_{\Delta_v}$ be the maximal compact subring in $\Delta_v$ and
we fix an isomorphism $\varphi_v: B_v \rightarrow
\text{Mat}_{m_v}(\Delta_v)$ such that  
$$R_v:= R\otimes_A O_v = \varphi_v^{-1}(\text{Mat}_{m_v}(O_{\Delta_v})).$$
Choose a generator $\Pi_v$ of the maximal ideal in $O_{\Delta_v}$.
As in the case when $\Delta_v$ is a field,
we have the \lq\lq Iwasawa\rq\rq\ decomposition for the units group
$\text{GL}_{m_v}(\Delta_v)$ 
of $\text{Mat}_{m_v}(\Delta_v)$,
i.e.\ every element in $\text{GL}_{m_v}(\Delta_v)$ can be written as the form
$$\begin{pmatrix}
    \Pi_v^{\ell_1} & u_{12} & \cdots & u_{1m_v}\\
    0 & \Pi_v^{\ell_2}  & \cdots &u_{2m_v}\\
    \vdots & \ddots & \ddots & \vdots \\
    0 & \cdots & 0 & \Pi_v^{\ell_{m_v}}
    \end{pmatrix} \cdot U,$$
where $\ell_1,\ldots,\ell_{m_v} \in \mathbb{Z}$, $u_{ij} \in \Delta_v$
for $1\leq i<j \leq m_v$, and  
the element $U$ is in $\text{GL}_{m_v}(O_{\Delta_v})$.
So for $\ell \geq 0$, $b(\mathfrak{p}_v^{\ell})$ is equal to the
number of representatives of the form 
$$\begin{pmatrix}
    \Pi_v^{\ell_1} & u_{12} & \cdots & u_{1m_v}\\
    0 & \Pi_v^{\ell_2}  & \cdots &u_{2m_v}\\
    \vdots & \ddots & \ddots & \vdots \\
    0 & \cdots & 0 & \Pi_v^{\ell_{m_v}}
    \end{pmatrix},$$
where $\sum_{i=1}^{m_v} \ell_i = \ell$, $\ell_i \geq 0$,
and $u_{ij} \in O_{\Delta_v}/\Pi_v^{\ell_i}O_{\Delta_v}$
for $1\leq i < j \leq m_v$.
This gives
$$b(\mathfrak{p}_v^{\ell}) = \sum_{\ell_1+\cdots + \ell_{m_v} = \ell
  \atop \ell_i \geq 0} 
  \bigg( \prod_{i=1}^{m_v} N(v)^{d_v \ell_i (m_v-i)}\bigg)
= \sum_{\ell_1+\cdots + \ell_{m_v} = \ell \atop \ell_i \geq 0}
  \bigg( \prod_{i=1}^{m_v} N(v)^{d_v \ell_i (i-1)}\bigg).$$
Hence
\begin{eqnarray}
\sum_{\ell=0}^{\infty}\frac{b(\mathfrak{p}_v^{\ell})}{N(v)^{\ell s}}
&=& \sum_{\ell=0}^{\infty}N(v)^{-\ell s}\left(\sum_{\ell_1+\cdots +
    \ell_{m_v} = \ell \atop \ell_i \geq 0} 
  \bigg( \prod_{i=1}^{m_v} N(v)^{d_v \ell_i (i-1)}\bigg)\right)
  \nonumber \\ 
  &=& \prod_{i=1}^{m_v}\left( \sum_{\ell_i
  =0}^{\infty}N(v)^{(i-1)d_v\ell_i-\ell_i s}\right) \nonumber \\ 
&=& \prod_{i=1}^{m_v}\frac{1}{(1-N(v)^{(i-1)d_v-s})}. \nonumber
\end{eqnarray}

\begin{thank}
  The authors would like to thank E.-U.~Gekeler and 
  Jing Yu for their steady interest and
  encouragements, A.~Schweizer for pointing out the
  reference \cite{denert-geel:classno88}, and the previous referees 
  for helpful comments. The authors were partially supported by the
  grants NSC 99-2119-M-002-007, NSC 97-2115-M-001-015-MY3 and
  AS-99-CDA-M01. 
\end{thank}

\end{document}